\documentclass[a4paper,12pt,reqno]{amsart}
\usepackage{amsmath,amsfonts,enumerate,amssymb,amsthm,color}
\usepackage{pgf,tikz}

\title{Skew-morphisms of cyclic $p$-groups}
\author{Istv\'an~Kov\'acs}
\address{I.~Kov\'acs, 
IAM and FAMNIT, University of Primorska, Glagolja\v ska 8, 
6000 Koper, Slovenia}
\email{istvan.kovacs@upr.si}
\author{Roman Nedela}
\address{R.~Nedela,
Institute of Mathematics, Slovak Academy of Sciences, Severn\'a 5, 
975 49 Bansk\'a Bystrica, Slovakia}  
\address{
 NTIS -- New Technologies for the Information Society, Faculty of Applied Sciences, University of West Bohemia,
Technick\'a 8, 306 14 Plze\v{n},
Czech Republic}

\email{nedela@savbb.sk}

\newtheorem{thm}{Theorem}
\newtheorem{lem}{Lemma}
\newtheorem{prop}{Proposition}

\theoremstyle{definition}

\theoremstyle{remark}

\theoremstyle{defintion}
\newtheorem*{prob}{Problem}

\def\Z{\mathbb{Z}}

\def\cN{\mathcal{N}}
\DeclareMathOperator{\Aut}{Aut}

\DeclareMathOperator{\Sym}{Sym}
\DeclareMathOperator{\Skew}{Skew}

\usepackage{color}

\begin{document}

\maketitle

\begin{abstract}
Let $G$ be a finite group having a factorisation $G=AB$ into  
subgroups $A$ and $B$ with $B$ cyclic and $A\cap B=1,$ 
and let $b$ be a generator of $B$.  The associated skew-morphism is the bijective mapping $f:A \to A$ well defined by the equality $baB=f(a)B$ where $a\in A$.   In this paper, we shall classify all skew-morphisms of cyclic $p$-groups where $p$ is an odd prime.
\end{abstract}

\section{Introduction}
Following Jajcay and \v{S}ir\'{a}\v{n} \cite{JS}, 
a \emph{skew-morphism} of a finite group $A$ is a bijective mapping 
$f:A\to A$ fixing the identity element of $A$ and having the property that $f(xy)=f(x)f^{\pi(x)}(y)$ for all 
$x,y\in A,$ where the integer $\pi(x)$ depends only on $x$. 
We also refer to the mapping $\pi : A \to \Z$ as a \emph{power function} corresponding to $f$. 
Although the concept of a skew-morphism 
was introduced and investigated in the context of regular Cayley maps and hypermaps, see \cite{CJT07,CT,JN15,JS}, 
Conder et al.\ \cite{CJT} pointed out that it appeared   
already in the context of factorisation of groups. Namely, let $G$ be a finite group having a factorisation $G=AB$ into subgroups $A$ and $B$ with $B$ cyclic and $A\cap B=1,$ and let $b$ be a generator of $B$. Then there exists a bijective mapping 
$f:A \to A$ well defined by the equality $baB=f(a)B$ where $a\in A$.  
It is not hard to show that $f$ is a skew-morphism of $A$. Moreover, 
all skew-morphisms of $A$ arise in this way. 
If, in addition, $A$ is normal in $G,$ then $bab^{-1}=\alpha(a)$ for some automorphism $\alpha$ of $A,$ and therefore, the 
skew-morphism $f=\alpha$. Every automorphism of $A$ is a 
skew-morphism; the converse, however, does not hold in general.
Several examples of skew-morphisms which are not 
group automorphisms were given in \cite{BJ,CJT,KN} for abelian 
and in \cite{CJT,ZD} for dihedral groups.  A problem of determining
all skew-morphisms for a given group arises. Since every automorphism 
is a skew-morphism, this problem is at least as hard as the problem of
determining $\Aut(G)$ for a given group $G$, which is not an easy task in general. One of the problems one needs to overcome consists in the fact that the composition of two skew-morphisms of $G$ may not be a skew-morphism. Even to determine the skew-morphisms of cyclic groups seems to be a difficult problem. 

\begin{prob}
Determine the skew-morphisms of cyclic groups.
\end{prob}

In this paper, we make a step towards its solution
by determining all skew-morphisms for cyclic $p$-groups, where $p>2$ is an odd prime. The next natural step to accomplish is a determination of skew-morphisms of cyclic $2$-groups. Partial solution of the above problem appears in \cite{CT}, where skew-morphisms containing a generating orbit closed under taking inverses are classified.

In \cite{KN}, we investigated skew-morphisms of cyclic groups in connection with Schur rings of cyclic groups. Among others we proved the following decomposition theorem.
In what follows we denote by $\Skew(\Z_n)$ the set of 
skew-morphisms of the cyclic additive group $\Z_n,$ and by  
$\phi$  Euler's totient function. 

\begin{thm}{\rm \cite[Theorem 1.1]{KN}} Let $n=n_1n_2$, where $\gcd(n_1,n_2)=1$, and $\gcd(n_1,\phi(n_2))=\gcd(\phi(n_1),n_2)=1$. Then $\sigma\in \Skew(\Z_n)$ if and only if $\sigma=\sigma_1\times \sigma_2$, where $\sigma_1\in \Skew(\Z_{n_1})$ and $\sigma_2\in \Skew(\Z_{n_2})$.
\end{thm}

It was shown in \cite[Corollary~4.10]{KN} (and reproved in \cite[Theorem~6.4]{CJT}) that the cyclic group $\Z_{p^2}$ has exactly $(p-1)(p^2-2p+2)$ skew-morphisms where $p$ is an odd prime. 
In this paper,  we generalise this result to the cyclic group $\Z_{p^e}$ of order $p^e$ with arbitrary $e\ge 2$ by proving

\begin{thm}\label{T-main}
If $e\ge  2$ and $p$ is an odd prime, then $\Z_{p^e}$ has 
exactly $(p-1)(p^{2e-1}-p^{2e-2}+2)/(p+1)$ skew-morphisms.
\end{thm}

Detailed description of the set $\Skew(\Z_{p^e})$ of skew-morphisms can be found in Theorems~\ref{T3} and \ref{T4},
which are the main results of the paper. 

The paper is organised as follows. In the next section we recall some known facts for further use.
Among others we state in Lemma~\ref{CJT} that the order of a 
skew-morphism of a cyclic group $\Z_n$ divides $n\phi(n)$. In particular, if $n=p^e$ for a prime $p$, the skew-morphisms split into two classes: the first class contains those skew-morphisms whose order is a power of $p$, the others form the complement of the first class. In Section~3 we define a two-parametrised family of skew-morphisms $s_{i,j}$ 
and investigate its properties. In Section~4 the above family is generalised to a family of skew-morphisms $s_{i,j,k,l}$ given by four integer parameters $i,j,k,l$. Finally, in the last section we prove
 that every skew-morphism in $\Skew(\Z_{p^e})$  is one of 
 $s_{i,j,k,l}$,
and that the skew-morphisms of the first class are exactly the skew-morphisms $s_{i,j}=s_{i,j,0,0}$, see Theorems~\ref{T3} and \ref{T4}. For the purpose of enumeration we clarify in Proposition~\ref{unique} which $4$-tuples of parameters determine the same skew-morphism.

\section{Preliminary results} 
Throughout the paper $\Z_{p^e}$ represents the cyclic group of order $p^e$ where $p$ is an odd prime. 
Let $\Sym(\Z_{p^e}),\, \Aut(\Z_{p^e})$ and $\Skew(\Z_{p^e})$ denote the set (group) of all permutations, automorphisms, and skew-morphisms of $\Z_{p^e},$ respectively. In this paper we multiply 
permutations from the right to the left, that  is, if $f$ and $g$ are 
in $\Sym(\Z_{p^e})$ and $x\in\Z_{p^e},$ then $(fg)(x)=f(g(x))$.
We set $t$ to denote the translation of $\Z_{p^e}$ defined by
$t(x)=x + 1$.  
\medskip

Let $s\in\Skew(\Z_{p^n})$ be a skew-morphism with power function $\pi$. The \emph{skew product group} of $\Z_{p^e}$ induced by $s$ is the group $\langle t,s\rangle$.  Notice that, the skew-product group 
factorises as $\langle t,s\rangle=\langle t\rangle \langle s\rangle,$ 
and for every $i\in\Z_{p^e},$
\begin{equation}\label{cr}
s t^{i}=t^{s(i)} s^{\pi(i)}.
\end{equation}
Also, $|\langle t,s\rangle|=p^e\cdot |s|,$ where $|s|$ is the order of the permutation $s$. The following 
converse also holds (see \cite[Proposition~2.2]{KN}): if  
$f\in\Sym(\Z_{p^e}),$ then 
\begin{equation}\label{criterion}
f(0)=0\;\text{and}\; |\langle t,f\rangle|=p^e\cdot |f| 
\implies  f\in\Skew(\Z_{p^e}).
\end{equation}

The next result is in \cite[Corollary 3.4]{KN}, see \cite[Theorem~6.1]{CJT}, as well.  

\begin{lem}\label{CJT}
Let $s\in\Skew(\Z_n)$. Then the order $|s|$ of $s$ divides 
$n\phi(n)$. Moreover, if $\gcd(|s|,n)=1$ or $\gcd(\phi(n),n)=1,$ then 
$s$ is an automorphism of $\Z_n$. 
\end{lem}

In the next lemma we collect a number of properties of skew-morphisms of $\Z_{p^e}$ proved in \cite{KN}. 
In fact, part (i) is \cite[Lemma~4.4(iii)]{KN}, (ii) is 
\cite[Lemma~4.4(ii)]{KN}, and (iii) is \cite[Theorem~4.1]{KN}.

\begin{lem}\label{KN}
Let $s\in\Skew(\Z_{p^e})$ be a skew-morphism of order $p^i$ with power function $\pi$. 
\begin{enumerate}[(i)]
\item The power $s^p$ is also in $\Skew(\Z_{p^e})$.
\item For all $x\in\Z_{p^e},$ $\pi(x)\equiv 1\pmod p$. In particular, 
if $s$ has order $p,$ then $s$ is in $\Aut(\Z_{p^e})$.
\item There exists an automorphism $\alpha\in\Aut(\Z_{p^e})$ of 
order $p^i$ such that the $\langle\alpha\rangle$-orbits  
coincide with the $\langle s\rangle$-orbits.  
\end{enumerate}
\end{lem}

It is well-known that $\Aut(\Z_{p^e})\cong \Z_{(p-1)p^{e-1}},$ and the automorphisms of order $p^i\, (1\le i\le e-1)$ are given as 
\begin{equation}\label{ord=pi}
x \mapsto (kp^{e-i}+1)x
\end{equation}
where $k\in\{1,\ldots,p^i-1\}$ and $p\nmid k$ (see, e.g.\ \cite{JJ}).
This implies that, for any $k\in\{0,\ldots,p^n-1\},$ 
\begin{equation}\label{gcd}
\gcd((p+1)^k-1,p^e)=p\cdot \gcd(k,p^{e-1}).
\end{equation}

\begin{lem}\label{KN2}
Let $s\in\Skew(\Z_{p^e})$ be a skew-morphism of order $p^i$. 
Then $\gcd(s(1)-1,p^e)=p^{e-i}$.
\end{lem}

\begin{proof} The statement is clear if $i=0,$ let $i\ge 1$.   
According to Lemma~\ref{KN}(iii), 
the orbit of $1$ under $\langle s\rangle$ is equal to the orbit 
of $1$ under $\langle\alpha\rangle$ for some automorphism 
$\alpha$ of order $p^i$.  By \eqref{ord=pi}, this orbit is  
$\Omega:=\{xp^{e-i}+1 : x\in\{0,\ldots,p^i-1\}\}$. 
This in turn implies that $\langle s\rangle$ is regular on $\Omega,$ 
and thus $s(1)$ is not in the orbit of $1$ under $\langle s^p\rangle$. 
Since $s^p$ is a skew-morphism, see Lemma~\ref{KN}(i), it follows 
that $s(1)\notin \{xp^{e-i+1}+1 : x\in\{0,\ldots,p^{i+1}-1\}\}$.  
The lemma is proved. 
\end{proof}

We end the section with a property of metacyclic $p$-groups. 

\begin{lem}\label{order}
Let $G$ be a metacyclic $p$-group given by the presentation 
$$
G=\big\langle x,y \mid x^{p^m}=y^{p^n}=1, x^y=x^{1+p^{m-n}}\big\rangle
$$
where $m > n\ge 1$. Then the order of the element $x^iy^j$ 
is equal to $\max\{|x^i|,|y^j|\}$ where  $|x^i|$ and $|y^j|$ denote 
the order of $x^i$ and $y^j,$ respectively.
\end{lem}

\begin{proof}
We prove the lemma by induction on $m$. If $m=2$ then 
$G$ is the unique non-abelian group of order $p^3$ and of 
exponent $p^2$. In this case the lemma follows by a straightforward 
computation. 

Let $m\ge 3$ and $N=\langle x^{p^{m-1}},y^{p^{n-1}} \rangle$. Notice that, $N\cong\Z_p^2,$ $N$ is normal in $G,$ and the factor group $G/N$ admits the presentation 
$$
G/N=\big\langle \bar{x},\bar{y} \mid 
\bar{x}^{p^{m-1}}=\bar{y}^{p^{n-1}}=1, \bar{x}^{\bar{y}}=
\bar{x}^{1+p^{m-n}}\big\rangle,
$$
where for $g \in G,$  $\bar{g}$ denotes the image of $g$ 
under the natural homomorphism $G\to G/N$. 
The lemma holds trivially if either $x^i=1$ or $y^j=1,$ hence 
we assume below that $x^i\ne 1$ and $y^j\ne 1$.
Then we find, using the induction hypothesis, that the order of $\bar{x}^i\bar{y}^j$ is equal to $\max\{|\bar{x}^i|,|\bar{y}^j|\}=\frac{1}{p}\max\{|x^i|,|y^j|\}$. Therefore, we are done if we show that 
$\bar{x}^i\bar{y}^j$ has order $\frac{1}{p}|x^iy^j|$ 
where $|x^iy^j|$ denotes the order of $x^iy^j$. 
Obviously, $x^iy^j\ne 1$. Let $z$ be an element in $\langle x^iy^j\rangle$ of order $p$.  If $z\notin N,$ then $\langle N,z\rangle$ is isomorphic to a group of order $p^3$ and of exponent $p$. This 
implies that $\langle N,z\rangle$ is not metacyclic, contradicting the 
fact that every subgroup of $G$ is metacyclic (cf. \cite[III.11.1]{H67}).
Therefore, $z\in N,$ and thus $\bar{x}^i\bar{y}^j$ has order 
$\frac{1}{p}|x^iy^j|,$ as it is required.
\end{proof}

\section{The skew-morphisms $s_{i,j}$}
For the rest of the paper let $a\in\Aut(\Z_{p^e})$ be the 
automorphism defined by  $a(x)=(p+1)x$. By \eqref{ord=pi}, $a$ has order  $p^{e-1}$. 
Consider the permutation $ta^j \in \Sym(\Z_{p^e})$ for some
$j\in\{0,1,\ldots,p^{e-1}-1\}$. Recall that 
$t$  is the translation $t(x)=x+1$. By Lemma~\ref{order}, $ta^j$  
has order $p^e,$ and thus it is a full cycle. Therefore, there 
exists a unique permutation $b_j\in\Sym(\Z_{p^e})$ such that $b_j(0)=0$ and the conjugate 
\begin{equation}
t^{b_j}:=b_jtb_j^{-1}=ta^j. 
\end{equation}
In fact, for $x\in \Z_{p^e}$ with $x\ne 0,$ the permutation $b_j$ can be expressed as
\begin{equation}\label{bj}
b_j(x)=1+(p+1)^j+\cdots+(p+1)^{(x-1)j}.
\end{equation}

Define the permutations
\begin{equation}\label{sij}
s_{i,j}=b_j^{-1} a^i  b_j, \quad i,j \in \{0,1,\ldots,p^{e-1}-1\}.
\end{equation}

\begin{prop}\label{sij=sm}
Every permutation $s_{i,j}$ defined in \eqref{sij} is a skew-morphism of 
$\Z_{p^e}$.
\end{prop}

\begin{proof} By definition the order $|s_{i,j}|=|a^i|$. 
Assume at first that $p\nmid i,$ or 
equivalently, $s_{i,j}$ has order $p^{e-1}$.
It is clear that $s_{i,j}(0)=0$. Thus by \eqref{criterion}, it is 
enough to show that $|\langle t,s_{i,j}\rangle|=p^{2e-1}$.
We have
$$
|\langle t,s_{i,j}\rangle| = |\langle t,s_{i,j}\rangle^{b_j}| =
|\langle t^{b_j},s_{i,j}^{b_j}\rangle| = |\langle ta^j,a^i\rangle| = p^{2e-1},
$$
therefore $s_{i,j}$ is a skew-morphism.

Now, suppose that $i=p^ki',\, p\nmid i'$. As $s_{i',j}$ is a skew-morphism of $\Z_{p^e},$ and $s_{i,j}=s_{i',j}^{p^k},$ it follows from 
Lemma~\ref{KN}(i) that $s_{i,j}$ is a skew-morphism too. 
\end{proof}

Notice that, the skew-morphism $s_{i,j}$ is not  uniquely determined by 
the parameters $i$ and $j$. For instance, $s_{0,j}$ is the identity mapping for every $j$. The rest of the section is devoted to the proof of following theorem.

\begin{thm}\label{T1}
Let $e\ge 2,$ and $i,i',j,j'\in\{0,1,\ldots,p^{e-1}-1\}$. Then
$$ 
s_{i,j}=s_{i',j'}\;\text{if and only if}\; i=i'\;\text{and}\;j \equiv j' \pmod { p^{e-2}/\gcd(i,p^{e-2})}.
$$ 
\end{thm}

The theorem will be derived in a sequence of lemmas.

\begin{lem}\label{bj-2}
Let $e\ge 2,$ and $j,j'\in\{0,1,\ldots,p^{e-1}-1\}$ such that
$j\equiv j'\pmod{p^u}$ for some $u\in\{0,1,\ldots,e-1\}$. Then
for all $x,x'\in\Z_{p^e}$ with $x\equiv x'\pmod{p^{e-1-u}},$
$$
b_{j'}(x)-b_j(x')=b_{j'}(x)-b_j(x).
$$
\end{lem}

\begin{proof} 
We divide the proof into four steps.
\medskip

\noindent {\bf Claim (a):} \ $b_j(p^{e-1-u})=b_{j'}(p^{e-1-u})$.
\smallskip

We start by the observation that if $x,y,z\in\Z_{p^e},$  then 
\begin{equation}\label{ijk}
(t^x a^y)^z=t^{x b_y(z)} a^{yz}.
\end{equation}

Then $(t a^j)^{p^{e-1-u}}=t^{b_j(p^{e-1-u})}
a^{jp^{e-1-u}},$ $(t a^{j'})^{p^{e-1-u}}=t^{b_{j'}(p^{e-1-u})} a^{j'p^{e-1-u}}$.
Since $a^{jp^{e-1-u}}=a^{j'p^{e-1-u}}$, Claim (a) is equivalent to
$(t a^j)^{p^{e-1-u}}=(t a^{j'})^{p^{e-1-u}}$.
Recall that, an arbitrary $p$-group $G$ is \emph{regular} if 
for all $x,y\in G,$ $(xy)^p=x^py^pc_1\cdots c_r,$ 
where all $c_i$ belong to the commutator $\langle x,y\rangle'$. In the case when $p>2,$ a sufficient condition for $G$ to be regular is that its commutator subgroup $G'$ is cyclic (see \cite[III.10.2)]{H67}). In particular, it follows that $\langle t,a\rangle$ is a regular $p$-group. 
Thus by \cite[III.10.6]{H67}, $(t a^j)^{p^{e-1-u}}=(t a^{j'})^{p^{e-1-u}}$ is equivalent to $( (t a^j)^{-1} t a^{j'})^{p^{e-1-u}}=1$.
Putting $j'-j=j_0p^u$, for some $j_0$, we get 
$(( t a^j)^{-1} t a^{j'})^{p^{e-1-u}}=
(a^{j_0p^u})^{p^{e-1-u}}=1$.
\medskip

\noindent {\bf Claim (b)} \ $b_j(xp^{e-1-u})=b_{j'}(xp^{e-1-u})$ for all
$x\in\{0,1,\ldots,p^{u+1}-1\}$.
\smallskip

Let $x\in\{0,1,\ldots,p^{u+1}-1\}$. 
Recall that $t^{b_j}=ta^j$. By this and \eqref{ijk}, 
\begin{eqnarray*}
(t^{xp^{e-1-u}})^{b_j}&=&(ta^j)^{xp^{e-1-u}}=t^{b_j(xp^{e-1-u})} a^{xp^{e-1-u}j} \\
(t^{xp^{e-1-u}})^{b_{j'}}&=&(ta^{j'})^{xp^{e-1-u}}=t^{b_{j'}(xp^{e-1-u})} a^{xp^{e-1-u}j'}.
\end{eqnarray*} 
Thus Claim (b) is equivalent to $(t^{xp^{e-1-u}})^{b_j}=(t^{xp^{e-1-u}})^{b_{j'}}$. Since this holds for $x=1$ by Claim (a), it also holds 
for all $x>1$.
\medskip

\noindent {\bf Claim (c)} \ $p^{e-1-u} \mid b_j(xp^{e-1-u})$ for all
$x\in\{0,1,\ldots,p^{u+1}-1\}$.
\smallskip

The order of $t a^j$ is $p^e$.
This implies that $(t a_j)^{xp^{e-u-1}}=t^{b_j(xp^{e-1})} a^{jxp^{e-u-1}}$ has order at most $p^{u+1}$.
As $a^{jxp^{e-u-1}}$ has order at most $p^u$, it follows by Lemma~\ref{order} that
$t^{b_j(xp^{e-u-1})}$ has order at most $p^{u+1}$. This yields (c).
\medskip 

\noindent {\bf Claim (d)} \ $b_{j'}(x')-b_j(x')=b_{j'}(x)-b_j(x)$ for all 
$x,x' \in \Z_{p^n}$ with $x \equiv x' \pmod{p^{e-1-u}}$.
\smallskip 

Put $x'=x+x_0p^{e-1-u}$ and $j'=j+j_0p^u$. By \eqref{bj} and (b),
\begin{eqnarray*}
b_{j'}(x')-b_{j'}(x) &=& b_{j'}(x)+(p+1)^{j'x}b_{j'}(x_0p^{e-1-u})-
b_{j'}(x) \\
&=& (p+1)^{jx} (p+1)^{j_0p^ux} b_{j'}(x_0p^{e-1-u}) \\
&=& (p+1)^{jx} (p+1)^{j_0p^ux} b_{j}(x_0p^{e-1-u}).
\end{eqnarray*}

By \eqref{ord=pi}, $(p+1)^{j_0p^ux} \equiv 1 \pmod {p^{u+1}}$. Thus we find, using also (c), that  $(p+1)^{j_0p^ux}$  $b_{j}(x_0p^{e-1-u})=b_{j}(x_0p^{e-1-u})$. Therefore,
$$ b_{j'}(x')-b_{j'}(x)=
(p+1)^{jx}b_{j}(x_0p^{e-1-u})=b_{j}(x')-b_{j}(x).$$
This completes the proof. 
\end{proof}

\begin{lem}\label{sijx}
Let $e\ge 2$ and $i\in\{0,\ldots,p^{e-1}-1\}$. Then for all 
$x\in\Z_{p^e},$   
$$
s_{i,j}(x)\equiv x\pmod {p\cdot\gcd(i,p^{e-1})}.
$$ 
\end{lem}
\begin{proof}
By \eqref{sij}, the order of $s_{i,j}$ is $p^{e-1}/\gcd(i,p^{e-1})$. 
By Lemma~\ref{KN}(iii), there exists an automorphism $\alpha \in\Aut(\Z_{p^e})$ of the same order such that the $\langle \alpha\rangle$-orbits coincide with the 
$\langle s_{i,j}\rangle$-orbits. It follows that 
$\alpha \in \langle a^{\gcd(i,p^{e-1})}\rangle$.
Let $x\in\Z_{p^e}$. Then there exists 
some $l\in\{0,\ldots,p^{e-1}-1\}$ with $\gcd(i,p^{e-1}) \mid l$ 
such that $s_{i,j}(x)=a^l(x),$ and hence
\begin{equation}\label{sijx2}
s_{i,j}(x)\equiv (p+1)^l x\pmod{p^e}.
\end{equation}
By \eqref{gcd}, $(p+1)^l\equiv 1\pmod{p\cdot\gcd(l,p^{e-1})}$. 
Since $\gcd(i,p^{e-1})\mid\gcd(l,p^{e-1}),$ it follows that 
$(p+1)^l \equiv 1\pmod{p\cdot\gcd(i,p^{e-1})}$. Substituting this 
in \eqref{sijx2}, the lemma follows.
\end{proof}

\begin{lem}\label{part1}
Let $e\ge 2,$ and $i,j,j'\in\{0,\ldots,p^{e-1}-1\}$. Then
\begin{equation}\label{iff2}
s_{i,j}=s_{i,j'}\; \text{if and only if}\; j \equiv j' \pmod { p^{e-2}/\gcd(i,p^{e-2})}.
\end{equation}
\end{lem}
\begin{proof}

The statement holds if $i=0$, and hence below we assume that 
$i>0$.

{\bf($\Rightarrow$)}
Let $s_{i,j}=s_{i,j'}$, and $\pi$ be the power function of $s_{i,j}$.
By \eqref{cr}, $s_{i,j} t=t^{s_{i,j}(1)}s_{i,j}^{\pi(1)}$ holds in
$\langle t,s_{i,j}\rangle$. This implies in $\langle t,s_{i,j}\rangle^{b_j}=\langle t a^j,a^i \rangle$,
$$ 
t^{(p+1)^i} a^{i+j}=a^i (t a^j)=(t a^j)^{s_{i,j}(1)}
(a^i)^{\pi(1)}=t^{b_j(s_{i,j}(1))}a^{s_{i,j}(1)j+\pi(1)i}.
$$
As $b_j (s_{i,j}(1))=(b_js_{i,j})(1)=(a^ib_j)(1)=(p+1)^i,$
we get 
\begin{equation}\label{aij}
a^{i+j}=a^{s_{i,j}(1)j+\pi(1)i}.
\end{equation}

Since $s_{i,j}=s_{i,j'}$, the same argument yields
$a^{i+j'}=a^{s_{i,j'}(1)j'+\pi(1)i}$. Thus 
$a^{(j-j')(s_{i,j}(1)-1)}=1,$ and hence
\begin{equation}\label{jj'}
j \equiv j' \pmod { \;p^{e-1}/\gcd(s_{i,j}(1)-1,p^{e-1}) \; }.
\end{equation}

By \eqref{sij}, the order of $s_{i,j}$ is  
$p^{e-1}/\gcd(i,p^{e-1})$. This and Lemma~\ref{KN2} yield 
$\gcd(s_{i,j}(1)-1,p^e)=p\cdot\gcd(i,p^{e-1})$. 
Since $i\ne 0,$ $\gcd(i,p^{e-1})=\gcd(i,p^{e-2})$ and 
$\gcd(s_{i,j}(1)-1,p^e)=\gcd(s_{i,j}(1)-1,p^{e-1}),$ and thus
$$
\gcd(s_{i,j}(1)-1,p^{e-1})=p\cdot\gcd(i,p^{e-2}).
$$ 
This and \eqref{jj'} yield the right side of \eqref{iff2}.
\medskip
 
{\bf ($\Leftarrow$)} 
Put $p^u=p^{e-2}/\gcd(i,p^{e-2})$. Then order of $a^i$ is 
$p^{e-1}/\gcd(i,p^{e-2})=p^{u+1}$, and hence by \eqref{ord=pi},
\begin{equation}\label{p+1}
(p+1)^i= zp^{e-u-1}+1\;\text{for some}\; z\in\{1,\ldots,p^{u+1}-1\}, \; p\nmid z.
\end{equation}

For two arbitrary mappings $f,g:\Z_{p^e}\to\Z_{p^e},$ their \emph{sum} $f+g,$ \emph{difference} $f-g,$ and \emph{product} 
$fg$ are  the mappings from $\Z_{p^e}$ to $\Z_{p^e}$ defined in 
the usual way, that is, for $x\in\Z_{p^e},$ 
$$
(f+g)(x)=f(x)+g(x),\;(f-g)(x)=f(x)-g(x),\;(fg)(x)=f(g(x)).
$$

Let $j'=j+j_0p^{u}$, for some $j_0$. For every $x\in \Z_{p^e}$ with 
$x\ne 0,$ 
$ \big(\, (p+1)^{j_0p^ux}-1 \,\big) \mid \big(\, (p+1)^{j'x}-(p+1)^{jx} \,\big)$.
Therefore, $p^{u+1}\mid\big(\, (p+1)^{j'x}-(p+1)^{jx} \,\big)$. 
This and \eqref{bj} imply 
$$
(b_{j'}-b_j)(x) \equiv 0 \pmod {p^{u+1}}\;\text{for all}\;x\in\Z_{p^e}.
$$ 
This and \eqref{p+1} yield
$a^i(b_{j'}-b_j)=b_{j'}-b_j$. 
Also, by Lemma~\ref{sijx}, $s_{i,j}(x)\equiv x\pmod {p^{e-u-1}}$ for all
$x\in\Z_{p^e}$. Therefore, by Lemma \ref{bj-2},
$$ 
(b_{j'}-b_j)(s_{i,j}(x))=(b_{j'}-b_j)(x)\;\text{for all}\; x\in\Z_{p^e}.
$$

Consequently, $(b_{j'}-b_j)s_{i,j}=b_{j'}-b_j=a^i(b_{j'}-b_j)$. Thus 
$$
a^ib_{j'}-a^ib_j=b_{j'}s_{i,j}-b_js_{i,j}=
b_{j'}s_{i,j}-a^ib_j;
$$
and we get $s_{i,j}=b_{j'}^{-1}a^ib_{j'}=s_{i,j'}$. 
\end{proof}

Theorem~\ref{T1} follows from Lemma~\ref{part1} and 
from the following lemma.

\begin{lem}\label{part2}
If $s_{i,j}=s_{i',j'},$ then $i=i'$.
\end{lem}
\begin{proof}
Let $s_{i,j}=s_{i',j'}$, and $\pi$ be the power function of $s_{i,j}$.
We prove the lemma by induction on the order of $s_{i,j}$. 
The statement holds obviously if $s_{i,j}$ is the identity permutation, 
that is, if $i=0$. Thus for the rest of the proof assume that $i\ne 0$.
By \eqref{aij}, 
$$ 
a^{(s_{i,j}(1)-1)j}=a^{i(1-\pi(1))}\;\text{and}\;
a^{(s_{i,j}(1)-1)j'} = a^{i'(1-\pi(1))}.
$$ 

Then $s_{ip,j}=s_{i,j}^p=s_{i',j'}^p=s_{i'p,j'}$. Thus by the 
induction hypothesis, $ip\equiv i'p\pmod{p^{e-1}}$. 
As $p\mid (\pi(1)-1),$ see Lemma~\ref{KN}(ii), we conclude
$i(1-\pi(1))=i'(1-\pi'(1))\pmod{p^{e-1}}$.
Thus the above equalities reduce to
$j\equiv j'\pmod{p^{e-1}/\gcd(p^{e-1},s_{i,j}(1)-1)}$. 
Since it has been already shown that 
$\gcd(s_{i,j}(1)-1,p^{e-1})=p\cdot\gcd(p^{e-2},i),$ it follows that 
$j\equiv j'\pmod{p^{e-2}/\gcd(p^{e-2},i)}$. 
Thus by Lemma~\ref{part1}, $s_{i',j'}=s_{i',j},$ and we get
$s_{i,j}=s_{i',j'}=s_{i',j}$. It is obvious that this implies that 
$i=i'$. 
\end{proof}

\section{The skew-morphisms $s_{i,j,k,l}$}
For the rest of the paper we set $b$ to be an automorphism of 
$\Z_{p^e}$ of order  $p-1$.
Define the permutations
\begin{equation}\label{sijkl}
s_{i,j,k,l}=b_j^{-1}a^ib^kb_lb_j, 
\end{equation}
where the integers $i,j,k,l$ satisfy the following
conditions
\begin{enumerate}
\item[(C0)] $i,l\in\{0,\ldots,p^{e-1}-1\}$, $k\in\{0,\ldots,p-2\}$, $j\in\{0,\ldots,p^{e-2-c}-1\}$, where 
$p^c=\gcd(i,p^{e-2});$  
\item[(C1)] if $i=0$ or $k=0,$ then $l=0;$
\item[(C2)] if $i\ne 0$ and $k\ne 0,$ then $p^c\mid j$ and 
$p^{\max\{c,e-2-c\}} \mid l$. 
\end{enumerate}

 A $4$-tuple $(i,j,k,l)$ of integers satisfying the conditions (C0), (C1) and (C2) will be called \emph{admissible}.

Notice that, the permutations $s_{i,j,k,l}$ include the skew-morphisms 
$s_{i,j}$. Namely, by \eqref{sij} and \eqref{sijkl}, 
$s_{i,j,0,0}=s_{i,j};$ and by 
Theorem~\ref{T1}, all skew-morphisms $s_{i,j}$ are obtained in this way.

Before we prove that all permutations $s_{i,j,k,l}$ are 
skew-morphisms we give two lemmas.

\begin{lem}\label{BC}
Let $G$ be a split metacyclic $p$-group given by the presentation 
$$
G=\big\langle x,y \mid x^{p^m}=y^{p^n}=1, x^y=x^{1+p^{m-n}}\big\rangle
$$ 
where $m > n\ge 1$. Then the automorphisms in $\Aut(G)$ are 
the mappings $\theta_{u,v,w},$ where 
$u,v \in \{0,\ldots,p^m-1\}$ with $p\nmid u,\, p^{m-n}\mid v,$ and 
$w\in \{0,\ldots,p^n-1\}$ with $p^{2n-m}\mid w$ whenever 
$2n>m,$ defined as 
$$
\theta_{u,v,w}(x^iy^j)=x^{ui+vj}y^{wi+j}\;\text{for all}\;
x^iy^j\in G.
$$ 
In particular, $|\Aut(G)|=(p-1)p^{m-1+n+\min\{n,m-n\}}$.
\end{lem}

\begin{proof}
This is a corollary of the more general result
\cite[Theorem~3.1]{BC}.
\end{proof}

\begin{lem}\label{theta}
With the notation of Lemma~\ref{BC}, let $u\in\{0,\ldots,p^m-1\},$ 
$p\nmid u$ such that $u\ne 1,$ and as a unit of the ring $\Z_{p^m},$ $u$ has order $d$ with $p\nmid d$. 
Then the automorphism $\theta_{u,0,w}$ has order $d$.
\end{lem}
\begin{proof}
A straightforward computation gives 
$$
(\theta_{u,0,w})^k(x^iy^j)=x^{u^ki}y^{wi(1+u+\cdots+u^{k-1})+j}.$$
Therefore, the order $|\theta_{u,0,w}| \ge d$. 
Also, $p^m\mid (u^d-1)$. On the other hand, since the order of 
the unit $u\ne 1$ is not divisible by $p,$  it follows that $p\nmid (u-1)$ as well. 
We conclude from these that $p^m\mid (1+u+\cdots+u^{d-1})$.  Using also that  $n<m,$ we find 
$(\theta_{u,0,w})^d(x^iy^j)=x^iy^j,$ and thus 
$|\theta_{u,0,w}|=d$.
\end{proof}

\begin{thm}\label{T2}
Every permutation $s_{i,j,k,l}$ defined in \eqref{sijkl} is a 
skew-morphism of $\Z_{p^e}$. 
\end{thm}

\begin{proof}
Let $s_{i,j,k,l}$ be a permutation defined in \eqref{sijkl}. 
If $k=0,$ then $s_{i,j,k,l}=s_{i,j}$ and the statement holds.
For the rest of the proof it is assumed that $k\ne 0$. 

Let $i=0$. Then $j=0$ by (C0) and $l=0$ by (C1). We get 
$s_{i,j,k,l}=b^k,$ which is an automorphism in $\Skew(\Z_{p^e})$.

Now, suppose that $i\ne 0$. 
Since $k\ne 0,$ $t^{b^k}=t^u$ for some 
$u\in\{2,\ldots,p^e-1\},$ $p\nmid u,$ and as unit of $\Z_{p^e},$ 
$u$ has order $d$ with $p\nmid d$. 
Let $p^c=\gcd(i,p^{e-2})$. 
Let $G=\langle ta^j, a^i,b^kb_l\rangle$ and 
$P=\langle ta^j,a^i \rangle$. It follows from (C2) that 
$P=\langle t,a^i\rangle$ and there exists $a_1\in\langle a^i\rangle$
such that $P$ admits the presentation
$$
P=\big\langle t,a_1 \mid t^{p^e}=(a_1)^{p^{e-1-c}}=1,  
t^{a_1}=t^{1+p^{c+1}}\big\rangle.
$$

Clearly, $a_1=a^{i'}$ for a some $i'$ satisfying 
$\gcd(i',p^{e-2})=p^c$. This and (C2) yield $p^{e-2}/\gcd(i',p^{e-2})\mid l,$  and so $s_{i',l}=s_{i',0}=a^{i'}$ follows from Theorem~\ref{T1}. In other words, $b_l$ commutes with $a_1=a^{i'}$. 
Now, we can write 
\begin{equation}\label{conj}
t^{b^kb_l}=t^ua^l\;\text{and}\; (a_1)^{b^kb_l}=a_1.
\end{equation}

The group $\langle a_1\rangle=\langle a^{p^c}\rangle$. 
On the other hand, by (C2), $p^c \mid l,$ and so 
$a^l=a_1^w$ for some $w\in \{0,\ldots,p^{e-1-c}-1\}$.
This and \eqref{conj} yield that $P$ is normal in $G$. 
We claim that $b^kb_l$ acts on $P$ by conjugation as the 
automorphism  $\theta_{u,0,w}$ described in Lemma~\ref{BC}. 
In fact, we have to show that $p^{e-2-2c}\mid w$ whenever 
$2(e-1-c)>e$. In order to see that this indeed holds,  observe that 
$a^l=a_1^w=a^{i'w},$ $\gcd(i',p^{e-2})=p^c,$ 
and finally $p^{e-2-c}\mid l,$ by (C2). 
 
Since $b^kb_l$ fixes $0$ and $P$ is transitive on $\Z_{p^e},$  
$Z_{\langle b^kb_l\rangle}(P)=1$. Equivalently, $b^kb_l$ acts
faithfully on $P,$ in particular, $|b^kb_l|=|\theta_{u,0,l}|$.
Now, by Lemma~\ref{theta}, $|b^kb_l|=d$. Thus 
$|G|=|P|\cdot d$. This implies that the stabilizer $G_0$ of $0$ in $G$ 
has order $|G_0|=p^{e-1-c}d$.
Also, as $a^i$ commutes with $b^kb_l,$ we find 
$|a^ib^kb_l|=p^{e-1-c}d=|G_0|,$ implying 
that $G_0=\langle a^ib^kb_l\rangle,$ and $G$ factorises as 
$G=\langle ta^j \rangle \langle a^ib^kb_l\rangle$. 
Therefore, the conjugate group $G^{(b_j)^{-1}}$ factorises as 
$G^{(b_j)^{-1}}=\langle t \rangle \langle s_{i,j,k,l}\rangle,$ 
and hence $s_{i,j,k,l}$ is a skew-morphism by \eqref{criterion}.
\end{proof}

\begin{prop}\label{unique}
Every permutation $s_{i,j,k,l}$ in \eqref{sijkl} is uniquely determined 
by the admissible 4-tuple $(i,j,k,l)$. In particular, there is a bijection between the set of admissible $4$-tuples of integers and the set of skew-morphisms
$s_{i,j,k,l}$.
\end{prop}

\begin{proof}
Suppose that $s_{i,j,k,l}$ and $s_{i',j',k',l'}$ are two skew-morphisms 
defined in \eqref{sijkl} for which $s_{i,j,k,l}=s_{i',j',k',l'}$. 
We are going to prove that $(i,j,k,l)=(i',j',k',l')$. It follows that these skew-morphisms have order 
$$
\frac{p^{e-1}(p-1)}{\gcd(i,p^{e-1})\gcd(k,p-1)}=
\frac{p^{e-1}(p-1)}{\gcd(i',p^{e-1})\gcd(k',p-1)}.
$$
In particular, $\gcd(i,p^{e-1})=\gcd(i',p^{e-1})$. 

Let $i=0$ or $i'=0$. Then we get $i=i'=0,$ hence $j=j'=0,$
and $l=l'=0,$ by (C1). Thus $b^k=s_{0,0,k,0}=s_{0,0,k',0}=b^{k'},$ implying that $k=k',$ and so $(i,j,k,l)=(i',j',k',l')$.

Now, suppose that none of $i$ and $i'$ is equal to $0$. 
Let $\gcd(i,p^{e-2})=p^c$ (so $\gcd(i',p^{e-2})=p^c$ as well). 
By \eqref{sij} and \eqref{sijkl},
$s_{i,j}b_j^{-1}b^kb_lb_j=s_{i',j'}b_{j'}^{-1}b^{k'}b_{l'}b_{j'}$. 
Both $s_{i,j}$ and $s_{i',j'}$ have order $p^{e-1-c} \ne 1,$ 
and both  $b_j^{-1}b^kb_lb_j$ and $b_{j'}^{-1}b^{k'}b_{l'}b_{j'}$ 
have the same order not divisible by $p$. We have proved above 
that $[s_{i,j},b_j^{-1}b^kb_lb_j]=[s_{i',j'},b_{j'}^{-1}b^{k'}b_{l'}b_{j'}]=1$ also hold, and we deduce from these that 
$s_{i,j}=s_{i',j'}$ and $b_j^{-1}b^kb_lb_j=b_{j'}^{-1}b^{k'}b_{l'}b_{j'}$. Since both $j$ and $j'$ are in $\{0,\ldots,p^{e-2-c}-1\},$ it follows by Theorem~\ref{T1} that $i=i'$ and $j=j'$. This implies that 
$b^kb_l=b^{k'}b_{l'}$. From this by \eqref{bj}, $b^k(1)=(b^kb_l)(1)=
(b^{k'}b_{l'})(1)=b^{k'}(1)$. This in turn implies that $k=k'$ and 
$b_l=b_{l'}$. Then, by \eqref{bj} again, 
$(p+1)^l=b_l(2)-1=b_{l'}(2)-1=(p+1)^{l'},$ from which $l=l'$. 
This completes the proof of the proposition.
\end{proof}

\section{Skew-morphisms of $\Z_{p^e}$}
In this section we prove that skew-morphisms $s_{i,j,k,l}$ 
comprise all skew-morphisms of $\Z_{p^e}$. Our argument is 
divided in two cases 
depending on whether the order of the skew-morphism is a 
$p$-power or not.
 
\subsection{Skew-morphisms of $p$-power order}
Let $G$ be a skew-product group of $\Z_{p^e}$ induced by 
a skew-morphism $s$ of some $p$-power order. 
Then $G$ factorises as $\langle t\rangle\langle s\rangle,$ hence 
by a result of Huppert \cite{H53} (cf. also \cite[III.11.5]{H67}) $G$ is metacyclic.  

\begin{lem}\label{L1}
Let $G=\Z_{p^e}\langle s\rangle$ be a skew product group of 
$\Z_{p^e}$ of order $p^{e+i}$, where $1\le i\le e-1$. If $G$ is  a split metacyclic group, then its commutator subgroup $G'$ has order $p^i,$ and the exponent $\exp(G/G')=p^{\max\{i,e-i\}}$. 
\end{lem}

\begin{proof} Since $G$ is a split metacyclic and non-abelian group ,
 it has a presentation in the form
$$
G=\langle x,y \mid x^{p^k}=1, y^{p^l}=1, x^y=x^{1+p^{k-m}}\rangle,
$$
where $1\le m\le\min\{l,k-1\}$.
The order $|G|=k+l,$ and hence $k+l=e+i$. The elements in $G$ 
of order $p$ generate the subgroup $\langle x^{p^{k-1}},y^{p^{l-1}}\rangle\cong\Z_p^2$. On the other hand, $G=\langle t\rangle \langle s\rangle$ where $s\in\Skew(\Z_{p^e}),$ and it can be easily seen that $Z(G)\cap \langle s\rangle=1.$ These imply that $Z(G)$ is cyclic. 
Therefore, $Z(G)\cap \langle y\rangle=1,$ from which $l=m<k$. 
We conclude that $\exp(G)=p^k,$ and thus $k=e,$ $i=l=m,$ 
$G'=\langle x^{p^{e-i}}\rangle,$  and finally, 
$\exp(G/G')=p^{\max\{i,e-i\}}$. 
\end{proof}

Let us consider the skew product groups of $\Z_{p^e}$ induced by 
the skew-morphisms $s_{i,j}$. Theorem~\ref{T1} implies that 
these groups can be listed as the groups $G(i,j),$ where 
$G(0,0):=\langle t\rangle;$ and for $e\ge 2,$ 
\begin{equation}\label{Gij}
G(i,j):=\langle t,s_{p^{e-1-i},j}\rangle,\; i\in\{1,\ldots,e-1\},\;j\in\{0,
\ldots,p^{i-1}-1\}.
\end{equation}
Notice that, $|G(i,j)|=p^{e+i}$. 

\begin{lem}\label{L2} 
Let $e\ge 2$ and $G(i,j)$ be a group defined in \eqref{Gij}.
Then $G(i,j)$ is a split metacyclic group if and only if 
$p^{e-1-i}\mid j$.
\end{lem}
\begin{proof}
Then  $G(i,j)^{b_j}=\langle ta^j,a^{p^{e-1-i}}
\rangle,$ see the proof of Proposition~\ref{sij=sm}. 
Let $G=\langle ta^j,a^{p^{e-1-i}}\rangle$.
This shows that, if $p^{e-1-i}\mid j,$ then $G$ is a split 
metacyclic group. It remains to prove that, if $G$ is a split metacyclic group, then $p^{e-1-i}\mid j$. In fact, we prove that 
$p^{e-1-i}\nmid j$ implies that $G$ is a non-split 
metacyclic group. 

Let $p^m=\gcd(j,p^{e-1}),$ and assume that $e-1-i > m$. By \eqref{Gij}, 
\begin{equation}\label{m}
0\le m\le i-2.
\end{equation}

The commutator $[x,y]$ of two elements $x,y \in G$ is given by 
$[x,y]:=xyx^{-1}y^{-1}$. 
Then $[a^{p^{e-1-i}},ta^j]=t^{(p+1)^{p^{e-1-i}}-1}$.
Let $N=\langle t^{(p+1)^{p^{e-1-i}}-1}\rangle$. Clearly, $N\le G'$. 
Notice that, $N$ is normal in $G,$ and $G/N$ 
is abelian. This implies that $N\ge G',$ and therefore, 
$G'=N$. By \eqref{gcd}, 
$|G'|=p^i$. On the other hand, 
$|\langle t\rangle \cap \langle ta^j\rangle|=p^{m+1}$. Thus $|G'\cap\langle ta^j\rangle|=p^{\min\{i,m+1\}},$ from 
which $|G'\cap\langle ta^j\rangle|=p^{m+1}$ by \eqref{m}. Also, 
$G'\le\langle t\rangle,$ hence $G'\cap\langle a\rangle=1,$ and so 
 $\exp(G/G')=p^{\max\{e-1-m,i\}}=p^{e-m-1}$ because we assumed 
$e-1-i > m$.  Finally, as $e-m-1$ is larger than both $i$ and $e-i,$ 
$\exp(G/G') > p^{\max\{i,e-i\}}$. Now, Lemma~\ref{L1} gives that 
$G$ is a non-split metacyclic group.  
\end{proof}

\begin{lem}\label{L3}
Let $e \ge 2$ and $1\le i\le e-1$.  
The number of non-isomorphic non-split skew product groups 
of $\Z_{p^e}$ of order $p^{e+i}$ is at most $\min\{i-1,e-1-i\}$.
\end{lem}
\begin{proof}
By \cite[Theorem~3.2]{K}, every metacyclic non-split $p$-group $G$  
has up to isomorphism a presentation in the form: 
\begin{equation}\label{reduced}
G=\langle x,y \mid x^{p^m}=1,y^{p^n}=x^{p^{m-u}},x^y=
x^{1+p^{m-c}} \rangle,
\end{equation}
where 
\begin{equation}\label{mnsc}
\max\{1,m-n+1\}\le u<\min\{c,m-c+1\}.
\end{equation}

Now, suppose that $G$ is a non-split skew product group 
of $\Z_{p^e}$ of order $p^{e+i}$ induced by some $s\in\Skew(\Z_{p^e})$.  
Consider the presentation of $G$ described in \eqref{reduced}. 
Since $G$ is non-split, $m<e$. The exponent $\exp(G)=p^e$. On the 
other hand, \eqref{reduced} shows that $\exp(G)=p^{\max\{m,n+u\}},$ and it follows that $e=n+u$. Thus the order $|y|=p^{n+u}=p^e,$ 
and we obtain $|Z(G)\cap\langle y\rangle|=p^{e-c}$.

We compute next $|Z(G)\cap\langle y\rangle|$ in another way.
Since $y$ has order $p^e,$ it acts on $\Z_{p^e}$ as a full cycle. 
In particular, $G=\langle y,s\rangle$. Suppose that $s$ centralizes 
$y_1\in \langle y\rangle$ . Then $s(y_1(0))=y_1(s(0))=y_1(0),$ and 
so $y_1(0)$ is fixed by $s$. Conversely, suppose that $s(t)=t$ for 
some $t\in\Z_{p^e}$. Then $t=y_2(0)$ for a unique $y_2\in\langle y\rangle,$ and we find  $[s,y_2^{-1}]$ fixes $0$. As $G'$ is normal and cyclic, it is semiregular on $\Z_{p^e}$. We conclude that 
$[s,y_2^{-1}]=1,$  and $y_2$ is centralized by $s$. By these we have shown 
$$
|Z(G)\cap\langle y\rangle|=|t\in\Z_{p^e} : s(t)=t|.
$$ 
Lemma~\ref{KN}(iii) implies 
$|t\in\Z_{p^e} : s(t)=t|=p^{e-i},$ and so 
$|Z(G)\cap\langle y\rangle|=p^{e-i}$ also holds; and therefore, 
$c=i$. From \eqref{reduced}, $|G|=p^{m+n}$. 
As $|G|=p^{e+i}$ and $|y|=p^{n+u}=p^e,$ we get 
$n=e-u$ and $m=i+u$. In view of \eqref{mnsc}, 
the number of possible groups in \eqref{reduced} is bounded above 
by the number of solutions of the inequalities 
$$
\max\{1,i-e+2u+1\}\le u<\min\{i,u+1\}
$$
with variable $u$. This number is $\min\{i-1,e-1-i\},$ and the lemma follows. 
\end{proof}

We are ready to prove the main result of the subsection.

\begin{thm}\label{T3}
The skew-morphisms of $\Z_{p^e}$ of $p$-power order are exactly 
the skew-morphisms $s_{i,j}$ defined in \eqref{sij}.
\end{thm}

\begin{proof} 
Let $s$ be a skew-morphism of order $p^i,$ and let 
$G=\langle t,s\rangle$. We show below that
$G$ is isomorphic to one of the groups $G(i,j)$ defined in \eqref{Gij}. 

Suppose that $G$ is a split metacyclic group. We have proved in Lemma~\ref{L1} that $\exp(G)=p^e,$ and this implies that 
$G\cong G(i,0)$. 

Now, suppose that $G$ is non-split. Then by Lemma~\ref{L3}, 
$2\le i\le e-2-i$. Furthermore, by Lemma~\ref{L2}, each of the groups 
$$
G(i,p^j),\;j\in \{0,\ldots,\min\{i-2,e-2-i\}\}
$$
is a non-split metacyclic group. Also, we have computed in 
the proof of Lemma~\ref{L2} that $\exp(G(i,p^j)/G(i,p^j)')=p^{e-1-j},$ and therefore, the above groups are pairwise non-isomorphic. 
This and Lemma~\ref{L3} yield that $G$ is isomorphic to one of 
the above groups $G(i,p^j)$.

Now, we may assume without loss of generality that $G$ is a 
subgroup of $\langle t,a\rangle$. Then the skew-morphism $s$ and its 
power function $\pi$ can be described as follows: there exists 
a factorisation $G=\langle x\rangle \langle y\rangle$ such that 
$|x|=p^e,\; |y|=p^i$ (thus $\langle x\rangle\cap\langle y\rangle=1$),  $Z(G) \cap \langle y\rangle=1,$ and for every $k\in \Z_{p^e},$
\begin{equation}\label{yxz}
yx^k=x^{s(k)}y^{\pi(k)}.
\end{equation}

Now, $x=t^ua^v$ and $y=t^wa^z$ for some 
$u,v,w,z\in\{0,\ldots,p^{e-1}-1\}$. 
Since $|x|=p^e,$ it follows that $p\nmid u,$ see Lemma~\ref{order}.
Suppose that $y$ is fixed-point free. 
This gives $((p+1)^z-1)X=-w$ has no solution for $X\in \Z_{p^e},$ 
or equivalently, $\gcd(w,p^e)<\gcd((p+1)^z-1,p^e)$. This in turn 
implies that $|t^w|>|a^z|,$ and $t^{p^{e-1}} \in \langle t^wa^z\rangle=\langle y\rangle$. As $t^{p^{e-1}} \in Z(\langle t,a\rangle),$ 
$Z(G)\cap\langle y\rangle\ne 1,$ a contradiction. Therefore, $y$ has 
a fixed point. It follows from this that $y$ can be mapped into 
$\langle a\rangle$ under conjugation by some element $x_1\in\langle x\rangle$. Thus $y^{x_1}=a^m$ for some $m\in\{0,\ldots,p^{e-1}-1\},$ and after conjugation by $x_1,$ \eqref{yxz} becomes 
$a^m x^k=x^{s(k)}(a^m)^{\pi(k)},$ and hence 
$$
a^m(t^ua^v)^k=(t^ua^v)^{s(k)}(a^m)^{\pi(k)}.
$$
This and \eqref{ijk} yield $t^{ua^m(b_v(k))})=t^{ub_v(s(k))}$. 
Since $p\nmid u,$ this implies $(a^mb_v)(k)=(b_vs)(k)$ for 
all $k\in\Z_{p^e},$ and so 
$a^mb_v=b_vs,$ from which $s=b_v^{-1}a^mb_v=s_{m,v}$. 
This completes the proof of the theorem.
\end{proof}

\subsection{Skew-morphisms whose order is not a $p$-power}
We complete the classification of all skew-morphisms of $\Z_{p^e}$ 
by proving the following theorem.

\begin{thm}\label{T4}
The skew-morphisms of $\Z_{p^e}$ whose order is not a $p$-power  are exactly the skew-morphisms $s_{i,j,k,l}$, 
$k\ne 0$, defined in \eqref{sijkl}.
\end{thm}

\begin{proof}
Let $s$ be a skew-morphism of order $p^cd,$ $p\nmid d$ and $d>1$. Notice that, $c\in\{0,\ldots,e-1\}$ and $d\mid (p-1)$ because of Lemma~\ref{CJT}. 

Let $c=0$. Then by Lemma~\ref{CJT}, $s$ is in $\Aut(\Z_{p^e})$ 
and it has order $d$. Thus $s=b^k$ for some 
$k\in\{2,\ldots,p-2\},$ and so $s=s_{0,0,k,0}$. 

For the rest of the proof it will be assumed that $c\ne 0$.
Write $s=s_1s_2$ where $s_1$ has order $p^c$ and $s_2$ has order $d$.  Let $G=\langle t,s\rangle$ and $P$ be a Sylow 
$p$-subgroup of $G$ containing $t$. The group $P$ factorises a 
$P=\langle t\rangle\langle s_1\rangle$. Now, \eqref{criterion} yields 
that $s_1$ is a skew-morphism of $\Z_{p^e}$ of order $p^c$. 
By Theorem~\ref{T3},  $s_1=s_{i,j}$ for some 
$i\in\{0,\ldots,p^e-1\}$ and 
$j\in\{0,\ldots,p^{c-1}-1\}$. In particular, 
$P=G(c,j)$. Since $|G:P|=d$ and $d\mid (p-1),$ by Sylow Theorems, 
$P$ is normal in $G$. The permutation $s_2$ fixes $0$ and $P$ 
is transitive on $\Z_{p^e}$. These yield that 
$Z_{\langle s_2\rangle}(P)=1,$ and thus $s_2$ acts by conjugation on  
$P$ as an automorphism of order $d$. We conclude, using the 
known fact (cf. \cite{BC}) that the automorphism group of a non-split metacyclic group is also a $p$-group, that $P=G(c,j)$ is a split 
metacyclic group. By Lemma~\ref{L2}, this is equivalent to the 
condition 
\begin{equation}\label{C2}
\gcd(i,p^{e-2})=p^{e-1-c} \mid j.
\end{equation}  

Let consider the conjugate group $G^{b_j}$. Then 
$G^{b_j}=\langle ta^j,a^i, s_3\rangle,$ where 
$s_3=s_2^{b_j},$ $P^{b_j}=\langle ta^j,a^i\rangle$ is normal in $G^{b_j},$ and $s_3$ acts on $P^{b_j}$ as an automorphism of 
order $d$. There exists $a_1\in\langle a^i\rangle$
such that $P^{b_j}$ admits the presentation
$$
P^{b_j}=\big\langle t,a_1 \mid t^{p^e}=(a_1)^{p^c}=1,  
t^{a_1}=t^{1+p^{e-c}} \big\rangle.
$$

Clearly, $a_1=a^{i'}$ for a some $i'$ satisfying 
$\gcd(i',p^{e-1})=\gcd(i,p^{e-1})=p^{e-1-c}$.
According to Lemma~\ref{BC} the element $s_3$ acts on 
$P^{b_j}$ by conjugation as an automorphism $\theta_{u,v,w},$ 
where $u$ is a unit of $\Z_{p^e},$  $v\in\{0,\ldots,p^e-1\},$  and $w\in\{0,\ldots,p^c-1\}$ such that $p^{2c-e}\mid w$ if $2c>e$.  
As $s_3$ commutes with $a_1,$ we find $a_1=a_1^{s_3}=
\theta_{u,v,w}(a_1)=t^va_1,$ hence $v=0$. Also, 
$t^{s_3}=\theta_{u,0,w}(t)=t^ua_1^w=t^ua^{i'w}$. 
This implies that $s_3=b^kb_{i'w}$ where $k\in\{0,\ldots,p-2\}$ and 
$k\ne 0$. Then $s^{b_j}=(s_1s_2)^{b_j}=(s_{i,j}s_2)^{b_j}=a^i s_3=a^ib^kb_{i'w},$ and thus $s=b_j^{-1}a^ib^kb_{i'w}b_j=
s_{i,j,k,l}$ where $l=i'w$. 

To finish the proof it remains to verify that 
the $4$-tuple $(i,j,k,l)$ is admissible. 
We have $i,l\in \{0,\ldots,p^e-1\}$ with $i\ne 0,$
$j\in\{0,\ldots,p^{c-1}-1\},$ and 
$k\in\{0,\ldots,p-2\}$ with $k\ne 0$. Now, $j$ belongs  to 
the required interval because $\gcd(i,p^{e-2})=p^{e-c-1},$ 
and we obtain that (C0) holds.

Since both $i\ne 0$ and $k\ne 0,$ we need to check whether 
(C2) holds. The first part follows from \eqref{C2}. The second 
part is equivalent to  $p^{\max\{e-1-c,c-1\}}\mid l$.  
Since $l=i'w$ and $\gcd(i',p^{e-1})=\gcd(i,p^{e-1})=p^{e-1-c},$ 
the divisibility $p^{e-1-c} \mid l$ follows.  
We are done if $e-1-c \ge c-1,$ thus suppose that $e-1-c < c-1$. 
In this case $2c>e,$ hence $p^{2c-e} \mid w,$ and we  get 
$p^{c-1} \mid i'w=l,$ as claimed. This completes the proof of 
the theorem. 
\end{proof}

\subsection{Enumeration}
Finally, we are ready to count the number of skew-morphisms.
In view of Proposition~\ref{unique} and Theorems~\ref{T3} and \ref{T4},  this is equivalent to count the number of admissible $4$-tuples $(i,j,k,l)$. Theorem~\ref{T-main} follows from the following theorem.

\begin{thm}\label{T5}
If $e\ge  2$ and $p$ is an odd prime, then the number 
of admissible $4$-tuples $(i,j,k,l)$ is equal to 
$(p-1)(p^{2e-1}-p^{2e-2}+2)/(p+1)$.
\end{thm}

\begin{proof}
Let $\cN_1$ denote the number admissible $4$-tuples 
$(i,j,k,l)$ with $k=0,$ and let $\cN_2$ denote the number of 
those with $k\ne 0$. By (C1), $\cN_1$ is equal to the number 
of admissible $4$-tuples $(i,j,0,0)$. Therefore, 
\begin{eqnarray*}
\cN_1&=& 1+\sum_{c=0}^{e-2}\sum_{i\in\{1,\ldots,p^{e-1}-1\} \atop \gcd(i,p^{e-2})=p^c}p^{e-2-c}=1+\sum_{c=0}^{e-2}(p^{e-1-c}-p^{e-2-c})p^{e-2-c} \\
&=& \frac{p(p^{2e-3}+1)}{p+1}.
\end{eqnarray*}

Furthermore, using (C0) and (C2), we find 
\begin{eqnarray*}
\frac{\cN_2}{p-2}&=& 1+\sum_{c=0}^{e-2}\sum_{i\in\{1,\ldots,p^{e-1}-1\} \atop \gcd(i,p^{e-2})=p^c}
p^{\max\{0,e-2-2c\}+e-1-\max\{c,e-2-c\}} \\
&=&1+\sum_{c=0}^{e-2}(p^{e-1-c}-p^{e-2-c})p^{e-1-c} 
=\frac{p^{2e-1}+1}{p+1}. 
\end{eqnarray*}

Therefore, $(p+1)(\cN_1+\cN_2)=
p(p^{2e-3}+1)+(p-2)(p^{2e-1}+1),$ and hence 
$\cN_1+\cN_2=(p-1)(p^{2e-1}-p^{2e-2}+2)/(p+1)$.
\end{proof}

\section*{Acknowledgements}
The first author was supported in part by 
the Slovenian Research Agency (research program P1-0285, and 
research projects N1-0032, N1-0038, J1-5433 and J1-6720). The second author was supported by the project LO1506 of the Czech Ministry of
Education, Youth and Sports.

\end{document}